\documentclass[11pt]{article}
\usepackage{a4wide}
\usepackage{enumerate}

\usepackage{amsmath, amsthm, amsfonts, amssymb, accents}
\usepackage{graphicx,color}

\usepackage{pstricks,pst-grad}
\usepackage{srcltx}
\usepackage{dsfont}

\theoremstyle{plain}
\newtheorem{theorem}{Theorem}[section]
\newtheorem{lemma}[theorem]{Lemma}
\newtheorem{corollary}[theorem]{Corollary}
\theoremstyle{remark}
\newtheorem{remark}[theorem]{Remark}

\def\tht{\theta}
\def\Om{\Omega}
\def\om{\omega}
\def\e{\varepsilon}
\def\g{\gamma}
\def\G{\Gamma}
\def\l{\lambda}
\def\p{\partial}
\def\D{\Delta}

\def\a{\alpha}
\def\b{\beta}

\def\si{\sigma}

\def\d{\delta}
\def\L{\Lambda}
\def\z{\zeta}

\def\vp{\varphi}

\def\Ho{\mathring{W}_2}

\def\Wl{W^{\mathrm{l}}}
\def\Ws{W^{\mathrm{s}}}

\def\Wloc{W^{\mathrm{loc}}}
\def\Wosc{W^{\mathrm{osc}}}

\def\di{\,d}

\def\iu{\mathrm{i}}

\newcommand{\be}{\begin{equation}}
\newcommand{\ee}{\end{equation}}
\def\la{\langle}
\def\ra{\rangle}

\def\Op{\mathcal{H}}
\def\fm{\mathfrak{h}}

\DeclareMathOperator{\RE}{Re}
 \DeclareMathOperator{\spec}{\sigma}

\DeclareMathOperator{\dist}{dist} \DeclareMathOperator{\supp}{supp}

\DeclareMathOperator{\Dom}{\mathfrak{D}}

     \newcommand{\NN}{\mathds{N}}
     \newcommand{\PP}{\mathbb{P}}
     
     \newcommand{\RR}{\mathds{R}}
     \newcommand{\ZZ}{\mathds{Z}}
     
     \newcommand{\cB}{\mathcal{B}}

     \def\cL{\mathcal{L}}

     \newcommand{\cS}{\mathcal{S}}

\numberwithin{equation}{section}

\newcommand{\hm}[1]{\leavevmode{\marginpar{\tiny%
$\hbox to 0mm{\hspace*{-0.5mm}$\leftarrow$\hss}%
\vcenter{\vrule depth 0.1mm height 0.1mm width \the\marginparwidth}%
\hbox to 0mm{\hss$\rightarrow$\hspace*{-0.5mm}}$\\\relax\raggedright #1}}}

\begin{document}
\allowdisplaybreaks

\title{Expansion of the spectrum in the weak disorder regime for random operators in continuum space}

\author{Denis Borisov$^{1}$, Francisco Hoecker-Escuti$^{2}$ and  Ivan Veseli\'c$^{3}$}

\maketitle

\begin{itemize}
\item[(1)] \emph{Institute of Mathematics Of Ufa Scientific Center of RAS, Chernyshevsky str., 112, Ufa, 450008, Russia}

 \emph{Department of Physics and Mathematics, Bashkir
State Pedagogical University, October rev. st.~3a, Ufa, 450000,
Russia}

\emph{Faculty of Science, University of Hradec Kr\'alov\'e
Rokitansk\'eho 62, 500 03, Hradec Kr\'alov\'e, Czech Republic}

\emph{E-mail:}\ \texttt{borisovdi@yandex.ru}, \\
\emph{URL:} \texttt{http://borisovdi.narod.ru/}

 \item[(2)] \emph{Technische Universit\"at Hamburg-Harburg, Am Schwarzenberg-Campus 3, 21073 Hamburg, Germany}
 \item[(3)] \emph{Faculty of Mathematics, Technische Universit\"at Dortmund, 44227 Dortmund, Germany},
 \emph{URL:}\ \texttt{http://www.mathematik.uni-dortmund.de/sites/lehrstuhl-ix}

\end{itemize}

\begin{abstract}
  We study the spectrum of random ergodic Schr\"odinger-type operators   in the weak disorder regime.
  We give upper and lower bounds on how much the spectrum expands at its bottom for very general perturbations.
  The background operator is assumed to be a periodic elliptic differential operator on $\RR^d$, not necessarily of second order.
\end{abstract}

\section{Introduction}

Quantum Hamiltonians modelling condensed matter typically exhibit a lattice structure. In the idealized setting one uses operators which are periodic { with respect to  } the lattice, while more realistic models, taking into account disorder present in the model, consist in random operators which are homogeneous and ergodic { with respect to  } lattice translations. In the framework of quantum mechanics the key  to understanding properties of the underlying physical system is the spectrum of the operator.  It is a fundamental fact in the theory of random operators that the spectrum of an ergodic or metrically-transitive ensemble of operators is almost surely non-random, in the sense that there exists a fixed set $\Sigma \subset \RR$ such that for almost all elements in the ensemble, the spectrum of the operator coincides with $\Sigma$.
This holds even in an abstract Hilbert space setting, cf. e.g. \cite{KirschM-82a}.

The structure or shape of the spectrum as a subset of the real axis and its measure theoretic features are intimately connected to conductance and transport properties of the condensed matter.
With the emergence of disorder the spectrum of as a set grows. In order to understand this expansion in a quantitative way, one introduces a scalar parameter coupling the random part of the Hamiltonian to the periodic background operator.  In the simplest case the coupling is linear.  The question of interest in the present paper is: At what rate does the spectrum expand, as the scalar coupling increases? This is the first step in the analysis of more delicate spectral properties.
Indeed, it is precisely the new portions of the spectrum generated by the random perturbation of the Hamiltonian where interesting phenomena (like a transition from point to continuous spectrum) are expected to occur. We consider here only the expansion of spectrum at its lowest edge, although it seems that a refinement of our methods together with substantially more technical effort would apply to general spectral edges as well.

The \emph{motivation for the present work} is the following:
The most commonly encountered situation for weak-disorder random Schr\"odinger operators
is that the spectrum expands linearly with the disorder --- provided that sign conspiracy does not prevent expansion altogether.
(If $-\Delta$ is perturbed by a non-negative potential, there is obviously no expansion of the spectrum.)
However, there are examples where the expansion is quadratic in the disorder parameter. So it is natural to ask whether even
\emph{slower growth} (cubic, quartic, \ldots) \emph{is possible to occur?}
In the most important case, growth can occur only either with a linear or a quadratic rate,
as we discuss in Corollary \ref{c:1}  and Remark \ref{r:1} below.

More generally, we derive in the present paper
estimates on the expansion of the spectrum which are in line  with earlier research,
in particular on the weak disorder regime, {e.~g.}
\cite{aizenman1994localization}, \cite{wang2001},
\cite{klopp2002a}, \cite{klopp2002b}, \cite{elgart2009lifshitz}, \cite{elgartetal2011locnonmono}, \cite{caoelgart2012}, \cite{hoeckerescuti2013}, \cite{hoeckerescuti2014},
\cite{borisovveselic2011}, \cite{borisovveselic2013}. Let us give some more details:
In \cite{aizenman1994localization} the fractional moment method is developed to show localization for the Anderson model in the weak disorder regime,
\cite{wang2001} treats the Anderson model as well and gives an estimate on the size of the energy interval where localization occurs as a function of the disorder parameter.
This bound has been refined in \cite{klopp2002a} and \cite{elgart2009lifshitz}, while a similar result for random operator on $\RR^d$ has been derived in
\cite{klopp2002b}.
In the later paper the single site potential is allowed to change sign, but needs to have a non-vanishing average.
In the context of the present paper, this means the spectral edge is shifted linearly as a function of the disorder parameter.
Discrete alloy type models on $\ZZ^d$ have been analyzed in \cite{elgartetal2011locnonmono},
and for $\ZZ^3$ an analogous result to \cite{klopp2002b} has been derived in \cite{caoelgart2012}. It covers also the case
when the average of the single site potential is zero, which requires a more detailed analysis.
In particular,  this means the spectral edge is shifted quadratically as a function of the disorder parameter.
A particular regime where one can prove Anderson localization is the Lifshitz tail region near the spectral bottom.
Upper and lower bounds on the size of the region (again as a function of the disorder parameter) have been given in
\cite{hoeckerescuti2013} and \cite{hoeckerescuti2014}. The results mentioned so far concern models where a random potential
is coupled linearly to a kinetic energy term. In contrast to this \cite{borisovveselic2011} and \cite{borisovveselic2013} analyse
low lying eigenvalues of (long, finite segments of) randomly shifted and curved waveguides, respectively.
In this case the random perturbation term depends in a nonlinear way on the random variables and
the disorder parameter and consist of (lower order) differential operators.

The present paper
provides a fundamental analysis about the location of the spectrum which is a prerequisite for identifying the energy  region
of Anderson localization.
For a broader discussion of the physical intuition and the relevance of our result to the general understanding of spectral properties of random Hamiltonians
we refer to the discussion in our previous work \cite{BorisovHEV}.
There we have carried out an analogous analysis for discrete Hamiltonians, i.e. matrix operators over $\ell^2(\ZZ^d)$. In the discrete setting one has less technical questions to take care of, for instance properly defined domains or compactness properties, either because certain operators are automatically bounded, or because auxiliary Hilbert spaces are finite dimensional.
Nevertheless, our approach carries over to a very general class of operators in continuum space. In the present paper we chose not to present the most general model class (which we may do in some later manuscript) but to restrict ourselves to the case that the unperturbed periodic part of the Hamiltonian is a differential operator, and the perturbations satisfy certain relative boundedness conditions. This framework has the advantage that it on one hand covers a variety of physically relevant cases, but on the other hand is specific enough to avoid a long list of abstract hypotheses.

We close the section
highlighting  the general scope, the flexible versatility, and the technical advancements of our approach.
In comparison to mentioned previous work on spectral properties of random Schr\"odinger type operators in the weak-disorder regime
the present paper has the following new features
\begin{itemize}
 \item The results are formulated for very general types of random operators (actually quadratic forms).

 \item {The results require only weak assumptions on the random variables entering the model. In particular, the range of values need not be of fixed sign, nor an interval.}

 \item {No monotonicity condition for the random perturbations is assumed.}

 \item {We prove both upper and lower bounds for the bottom of the perturbed spectrum.}
\end{itemize}
This is motivated by the fact that the mechanism which determines the expansion of the spectrum is
of quite universal nature. So it makes more sense to establish it once and for all models of practical importance,
rather than to repeat the argument in each setting separately.

\section{Model and main results}

Let $x=(x_1,\ldots,x_d)$ be Cartesian coordinates in $\mathds{R}^d$, $d\geqslant d_1\geqslant 1$, $e_1,\ldots,e_{d_1}$ be linearly independent vectors in $\mathds{R}^d$, and $\G$ be the lattice $\{x\in\mathds{R}^d:\, x=z_1 e_1 +\ldots + z_{d_1}e_{d_1},\,z_i\in \mathds{Z}\}$. By $\Pi$ we denote an infinite domain in $\mathds{R}^d$ which is invariant  with respect to   shifts along $\G$, namely, for each $x\in\Pi$, $k\in\G$ we have $x-k\in\Pi$.
Let $\square$ be a periodicity cell of $\Pi$, i.e., a minimal domain such that $\Pi$ is the interior of $\overline{\bigcup\limits_{k\in\G} \square_k}$, $\square_k:=\{x\in\mathds{R}^d:\, x-k\in \square \}$.

Let $m\geqslant 1$ be a given natural number, 
$A_{\a\b}=A_{\a\b}(x)$, $\a,\b\in\mathds{Z}_+^d$, $|\a|,|\b|\leqslant m$, be functions defined on $\Pi$ satisfying the conditions:
\begin{equation}\label{2.1}
\begin{aligned}
&A_{\a\b}\ \text{are $\square$-periodic},\quad A_{\a\b}\in C^{|\a|}(\overline{\Pi}), 
\quad \overline{A_{\b\a}}=A_{\a\b},
\\
&\sum\limits_{\genfrac{}{}{0 pt}{}{\a,\b\in\mathds{Z}_+^d}{|\a|=|\b|=m}} A_{\a\b}(x)\xi^{\a+\b} 
\geqslant c_0
|\xi|^{2m},\quad x\in\overline{\Pi},\quad \xi\in\mathds{R}^d,
\end{aligned}
\end{equation}
where for $\xi=(\xi_1,\ldots,\xi_d)$, $\a+\b=(\a_1+\b_1,\ldots,\a_2+\b_2)$ we denote $\xi^{\a+\b}:=\xi_1^{\a_1+\b_1}\cdot\ldots\cdot\xi_d^{\a_d+\b_d}$,
and $c_0$ is a fixed positive constant independent of $x$ and $\xi$.
The boundary of the domain $\Pi$ is assumed to be $C^{2m}$-smooth.
In $L_2(\Pi)$ we consider the operator
\begin{equation}\label{2.2}
\Op_0:=  \sum\limits_{
\genfrac{}{}{0 pt}{}{\a,\b\in \mathds{Z}_+^d}{|\a|,|\b|\leqslant m}
} (-1)^{|\a|}\p^\a A_{\a\b} \p^\b,\quad \p^\a:=\frac{\p^\a\,}{\p x^\a} 
\end{equation}
subject to the Dirichlet condition
\begin{equation}\label{2.3}
u=0,\quad \frac{\p^j u}{\p\nu^j}=0,\quad \text{on}\quad \p\Pi,\quad j=1,\ldots,m-1,
\end{equation}
where $\nu$ is the outward normal to $\p\Pi$.
Thanks to  conditions (\ref{2.1}), the operator $\Op_0$ is self-adjoint, elliptic, and lower  semi-bounded
on the domain
\begin{equation*}
\Dom(\Op_0):=\{u\in W_2^{2m}(\Pi):\, \text{boundary conditions (\ref{2.3}) are satisfied}\}.
\end{equation*}
The associated sesquilinear form is
\begin{equation}\label{2.4}
\fm_0(u,v)= \sum\limits_{ \genfrac{}{}{0 pt}{}{\a,\b\in\mathds{Z}_+^d}{|\a|,|\b|\leqslant m}} (A_{\a\b} \p^\b u , \p^\a v )_{L_2(\Pi)} \quad\text{in}\quad L_2(\Pi) 
\end{equation}
on the domain
\begin{equation*}
\Dom(\fm_0):=\{u\in W_2^{m}(\Pi):\, \text{boundary conditions (\ref{2.3}) are satisfied}\}.
\end{equation*}
The above stated facts about $\Op_0$ can be proven in the same way as for second order differential operators,
one should just  employ appropriate smoothness improving theorems, see \cite[Ch. III, Sect. 6, Lm. 6.3]{Ber}.

Let  $\omega_k$, $k\in\G$, be a sequence  of bounded, non-trivial, independent, identically distributed random variables with distribution measure $\mu$. We assume that
\begin{equation*}
\{s_-,s_+\}\in \supp\mu \subseteq [s_-,s_+],
\end{equation*}
where the
numbers $s_\pm$ satisfy one of the following alternatives:
\begin{equation}\label{2.7}
s_-<0<s_+
\end{equation}
or
\begin{equation}\label{2.8}
0\leqslant s_-<s_+.
\end{equation}
Without loss of generality we suppose that each random variable $\omega_k$ is defined on the probability space $(\supp \mu,\mathfrak{B},\mu)$, where $\mathfrak{B}$ is the Borel $\si$-algebra on $\supp \mu\subset [s_-,s_+]$,
and $\omega_k$ are just the identity mappings on $\supp \mu$.
Thus the random field  $(\omega_k)_{k\in\G}$ 
is an element of the product probability space $(\Om, \bigotimes_\G \mathfrak{B}, \PP)$ with
measure $\PP:=\bigotimes_{k\in\G} \mu$ and configuration space $\Om:=\times_{k\in\G}\supp \mu$.

We let $\g_\Pi:=\p\square\cap\p\Pi$, $\g_l:=\p\square\setminus\p\Pi$.
By $\Ho^{2m}(\square,\g_\Pi)$ we denote the subspace of $W_2^{2m}(\square)$ consisting of functions satisfying the boundary conditions
\begin{equation}\label{2.5}
u=0,\quad \frac{\p^j u}{\p\nu^j}=0\quad \text{on}\quad \g_\Pi,\quad j=1,\ldots,m-1.
\end{equation}
Let $\cL(t)$, $t\in[-t_0,t_0]$, $t_0>0$, be a family of operators from $\Ho^{2m}(\square,\g_\Pi)$ into $L_2(\square)$ defined as
\begin{equation}\label{2.6}
\cL(t)=t \cL_1 + t^2 \cL_2 + t^3 \cL_3(t),
\end{equation}
where $\cL_i$ are bounded symmetric operators from $\Ho^{2m}(\square,\g_\Pi)$ into $L_2(\square)$;
moreover,  the operator $\cL_3(t)$ is assumed to be bounded uniformly in $t$.
We also suppose that the map
$$
t \mapsto \left(\cL_3(t): \Ho^{2m}(\square, \g_\Pi)\to L_2(\square)\right)
$$
is continuous in $t\in[-t_0,t_0]$.

By $\cS(k)$, $k\in\G$, we denote the shift operator: $(\cS(k)u)(x):=u(x-k)$.
The main operator of our study is
\begin{equation}\label{2.8a}
\Op_\e(\om):=\Op_0 + \cL_\e(\om),
\quad \om:=(\omega_k)_{k\in\G},\quad \cL_\e(\om):=\sum\limits_{k\in\G} \cS(k) \cL(\e\omega_k) \cS(-k),
\end{equation}
in $L_2(\Pi)$ on the domain $\Dom(\Op_\e(\om)):=\Dom(\Op_0)$. Here $\e$ is a small positive parameter.

Let us clarify the action of the operator $\cL_\e(\om)$. Given $u\in\Dom(\Op_0)$, it is clear that the restriction of $u$ on $\square_k$ 
belongs to $\Ho^{2m}(\square_k,\g_\Pi^k)$, $\g_\Pi^k:=\p\square_k\cap\p\Pi$. Then $(\cS(-k)u)(x)=u(x+k)$ is an element of $\Ho^{2m}(\square,\g_\Pi)$ and the action of $\cL(\e\omega_k)$ on $\cS(-k)u$ is well-defined as an element of $L_2(\square)$.
By applying $\cS(k)$ to the result of the action, we just shift the function $\cL(\e\omega_k)\cS(-k)u$ to  the cell $\square_k$.
In this way each operator $\cS(k) \cL(\e\omega_k) \cS(-k)$ acts on $\Ho^{2m}(\square_k,\g_\Pi^k)$
and the operator $\cL_\e(\om)$ is a sum of such single cell actions.
Since the operators $\cL_i$ are $\Op_0$-bounded and symmetric, by  the Kato-Rellich theorem  the operator $\Op_\e(\om)$ is self-adjoint for sufficiently small $\e$.

Our main aim is to study the behavior of the spectrum $\spec(\Op_\e(\om))$ of the operator $\Op_\e(\om)$.
The later results will consider the asymptotic behaviour  for (very) small $\e>0$.
For our first main result we merely require that $\e$ is sufficiently small so that for all $\omega$ the
perturbed operator $\Op_\e(\om)$ is self-adjoint.

\begin{theorem}\label{th2.1}
For all sufficiently small $\e$ there exists a closed set $\Sigma_\e$ such that
\begin{equation}\label{2.9}
\spec(\Op_\e(\om))=\Sigma_\e\quad\PP-a.s.
\end{equation}
The set $\Sigma_\e$ is equal to the closure of the union of spectra of periodic realizations of $\Op_\e(\om)$, explicitly,
\begin{equation}\label{2.10}
\Sigma_\e=\overline{\bigcup\limits_{N\in \NN} \quad \bigcup\limits_{\xi \ \text{is $2^N \G$-periodic}} \spec\big(\Op_\e(\xi)\big)},
\end{equation}
where the second union is taken over all sequences $\xi: \G\to\supp \mu $, which are periodic  with respect to   the sublattice $2^N \G:=\{2^N q:\, q\in\G\}$.
\end{theorem}
Here we adopt the following convention: In statements which are deterministic, i.e.~valid for all configurations
$\xi\colon \G \to \supp\mu$ we will denote the sequences by $\xi$, in statements which are probabilistic, e.g.~hold only for almost all configurations, we will use the symbol $\omega$ for the configuration $\omega\colon \G \to \supp\mu$.

The next part of our results is devoted to the position of $\inf \Sigma_\e$.
More precisely, we shall describe how $\inf \spec(\Op_0)$ is shifted by the perturbation $\cL_\e(\om)$.
First we describe the spectrum of $\Op_0$. Since this operator is periodic, we can employ a Floquet-Bloch decomposition to find $\spec(\Op_0)$. We introduce the Brillouin zone
\begin{equation*}
\square^*:=\{\theta\in\la e_1,\ldots,e_{d_1}\ra:\, \theta=\theta_1^* e_1^*+\ldots +\theta_{d_1}^* e_{d_1}^*,\, \theta_i^*\in[0,1)\},
\end{equation*}
where $e_1^*,\ldots,e_{d_1}^*$ are the vectors in the linear span $S:=\la e_1,\ldots,e_{d_1}\ra$ defined by the conditions
\begin{equation*}
(e_i,e_j^*)_{\mathds{R}^d}=2\pi\d_{ij},
\end{equation*}
and $\d_{ij}$ is the Kronecker delta.
On $\square$ we introduce the operator
\begin{equation}\label{2.11}
\Op_0(\theta):=
\sum\limits_{
\genfrac{}{}{0 pt}{}{\a,\b\in \mathds{Z}_+^d}{|\a|,|\b|\leqslant m}
} (-1)^{|\a|} (\p + i \theta)^\a A_{\a\b} (\p + i \theta)^\b  
\end{equation}
subject to boundary conditions (\ref{2.5}) and to periodic boundary conditions on $\g_l$.
Here, $\iu$ is the imaginary unit and for a given $\a=(\a_1,\ldots,\a_d)\in \mathds{Z}_+^d$, $\theta=(\theta_1,\ldots,\theta_d)\in S$, the symbol $(\p + i \theta)^\a$ stands for the differential expression
\begin{equation*}
\left(\frac{\p\,}{\p x_1}+ i\theta_1\right)^{\a_1} \left(\frac{\p\,}{\p x_2}+ i\theta_2\right)^{\a_2}\cdot\dots\cdot\left(\frac{\p\,}{\p x_d}+ i\theta_d\right)^{\a_d}.
\end{equation*}
The spectrum of $\Op_0$ is given by the identity \cite[Ch. 4, Sect. 4.5, Thm. 4.5.1]{K93}
\begin{equation}\label{2.12}
\spec(\Op_0)=\bigcup\limits_{\theta\in\square^*} \spec(\Op_0(\theta)).
\end{equation}
We assume that $\inf\spec(\Op_0(\theta))$ is a discrete eigenvalue for all $\theta\in\square^*$.
We denote this eigenvalue by $E_0(\theta)$.
The function $\theta\mapsto E_0(\theta)$ is $\square^*$-periodic and continuous in $\overline{\square^*}$, a compact set.
Consequently it attains the global minimum at some point $\theta_0\in\square^*$:
\begin{equation}\label{2.14}
\L_0:=E_0(\theta_0)=\min\limits_{\square^*} E_0(\theta).
\end{equation}
There is a positive constant  $c_1$ independent of $\theta\in\square^*$ such that
\begin{equation}\label{2.13}
\dist\Big(E_0(\theta),\spec(\Op_0(\theta))\setminus\{E_0(\theta)\}\Big)\geqslant c_1,
\end{equation}
Again by compactness and formula (\ref{2.12})
\begin{equation}\label{2.15}
\inf\spec(\Op_0)=\L_0.
\end{equation}
The eigenfunctions of $\Op_0(\theta_0)$ associated with $E_0(\theta_0)$ and orthonormalized in $L_2(\square)$ are denoted by $\psi_0^{(j)}$, $j=1,\ldots,n$.
We choose these eigenfunctions so that the matrix with the entries
$(\cL_1 e^{\iu \theta_0 x}\psi_0^{(i)},e^{\iu \theta_0 x}\psi_0^{(j)})_{L_2(\square)}$ is diagonal.
This  is possible by the theorem on simultaneous diagonalization of two quadratic forms in a finite-dimensional space.

Consider the $2n $ numbers
\begin{equation*}
s_-(\cL_1 \psi_0^{(i)},\psi_0^{(i)})_{L_2(\square)},\quad i=1,\ldots,n, \quad s_+(\cL_1 \psi_0^{(i)},\psi_0^{(i)})_{L_2(\square)},\quad i=1,\ldots,n.
\end{equation*}
Assume that the minimal value is attained at an index $i=i_0$ and for $s_*\in\{s_-,s_+\}$. We denote
\begin{equation}\label{2.16}
\L_1:=(\cL_1e^{\iu \theta_0 x}\psi_0,e^{\iu \theta_0 x}\psi_0)_{L_2(\square)},\quad \psi_0:=\psi_0^{(i_0)}.
\end{equation}
By $\psi_1$ we denote the unique solution to the equation
\begin{equation}\label{2.17}
(\Op_0(\theta_0)-\L_0) \psi_1=-e^{-\iu \theta_0 x} \cL_1 e^{\iu \theta_0 x}\psi_0+\L_1\psi_0
\quad 
\perp \ker  (\Op_0(\theta_0)-\Lambda_0) 
\end{equation}
such that this solution is orthogonal to $\psi_0^{(j)}$, $j=1,\ldots,n$. We let
\begin{equation}\label{2.18}
\L_2:=(\cL_1e^{\iu \theta_0 x}\psi_1,e^{\iu \theta_0 x}\psi_0)_{L_2(\square)} +(\cL_2e^{\iu \theta_0 x}\psi_0,e^{\iu \theta_0 x}\psi_0)_{L_2(\square)}.
\end{equation}
Our next main result reads as follows.
\begin{theorem}\label{th2.2}
For any minimizing triple $\theta_0\in\square^*$, $i_0\in\{1, \ldots,n\}$, $s_*\in\{s_-,s_+\}$
as defined above, there exist $C,\e _0 \in(0,\infty)$ such that for all $\e \in [0, \e _0]$
and for all sequences $\xi\colon \Gamma \to [s_-,s_+]$
\begin{equation}\label{2.19}
\spec(\Op_\e(\xi))\subseteq\big\{\l\in\mathds{R}:\, \dist(\l,\spec(\Op_0))\leqslant C\e(|\l|+1)\big\}.
\end{equation}
For the bottom of the almost sure spectrum $\Sigma_\e$ the estimate
\begin{equation}\label{2.20}
\Sigma_\e
\leqslant \L_0+
\e s_* \L_1 + \e^2 s_*^2 \L_2 + \frac{\e^3 s_*^3 \L_3(\e s_*)}{1+\e^2 s_*^2 \|\psi_1\|_{L_2(\square)}^2}
\end{equation}
holds true, where $\L_3 \colon \RR \to \RR$ is a continuous function, in particular, uniformly  bounded on compact intervals.
\end{theorem}
In fact, one can show the following explicit representation for $\L_3(t)$:
\begin{equation}\label{2.21}
\begin{aligned}
\L_3(t)=  &-(\L_1+t\L_2)\|\psi_1\|_{L_2(\square)}^2
+2 \RE ( \cL_2 e^{\iu \theta_0 x}\psi_0,  e^{\iu \theta_0 x}\psi_1)_{L_2(\square)}\\
 & + \big( (\cL_1+t\cL_2) e^{\iu \theta_0 x}\psi_1,  e^{\iu \theta_0 x}\psi_1\big)_{L_2(\square)}
\\
&+\big(\cL_3(t) e^{\iu \theta_0 x}\psi_0 (\psi_0+t\psi_1),e^{\iu \theta_0 x}(\psi_0+t\psi_1)\big)_{L_2(\square)}.
\end{aligned}
\end{equation}

\begin{remark}\label{rm2.2}
If we have several minimizing triples $\{\tht_0,i_0,s_*\}$, estimate (\ref{2.20}) is valid for each of them.
In this situation one should choose a triple, for which $s_* \L_1$ is minimal.
If all such quantities are same, then one should minimize $s_*^2\L_2$.
\end{remark}


For many types of periodic operators $\Op_0$ and generic perturbations $\cL_1$ the coefficient $\L_1$ does not vanish,
and one sees a linear shift of $\Sigma_\e$. Let us consider the case $\L_1=0$.
We will show that in this case, for many models,
the almost sure spectrum $\Sigma_\e$ must expand at least quadratically
 ---  i.e.~cubic, quartic, or weaker expansions cannot occur.

\begin{corollary}\label{c:1}
Assume that $\L_1=0$ and $\cL_2$ is a non-positive operator, cf.~(\ref{2.6}) and (\ref{2.16}).
Then
\[
\L_2 \leqslant - c_1 \|\psi_1\|_{L_2(\square)}^2
\]
where $c_1$ is the strictly positive spectral gap in (\ref{2.13}).
\end{corollary}


\begin{remark} \label{r:1}
Thus the question arises, in what situations we can ensure that $\psi_1$ does not vanish,
i.e.~that $e^{-\iu \theta_0 x} \cL_1 e^{\iu \theta_0 x}\psi_0$ is not the zero vector.
This is for instance the case if $\Op_0$ is the pure Laplacian and $\cL_1$ a multiplication operator $V$.
In this case $\theta_0=0$, $\Op_0(0)$ is the Laplacian with periodic boundary conditions, $\psi_0$
is the normalized constant function, so that $V\psi$ is only identically zero if $V$ itself is trivial.
More generally, $\Op_0$ can be any operator satisfying a unique continuation property for eigenfunctions
or the Harnack inequality for the ground state.

Likewise, $\psi_1$ cannot vanish if $\Op_0$ satisfies the Harnack inequality and
$e^{-\iu \theta_0 x} \cL_1 e^{\iu \theta_0 x}$ is a positivity preserving operator.\footnote{An $L_2$-function is called positive if $f$ is nonnegative almost everywhere and does not
vanish identically. A self-adjoint operator $A$ on an $L_2$-space is called positivity preserving if $Af$ is positive
whenever $f$ is in the domain of $A$ and positive.}
\end{remark}

Our next result provides a lower bound for $\inf\Sigma_\e$.
First we need to introduce additional notations and assumptions:
Given $u,v\in \Ho^{2m}(\square,\g_\Pi)$, by integration by parts it is easy to convince oneself that
\begin{equation}\label{2.23}
\begin{aligned}
\sum\limits_{
\genfrac{}{}{0 pt}{}{\a,\b\in \mathds{Z}_+^d}{|\a|,|\b|\leqslant m}
} & (-1)^{|\a|} \big( (\p + \iu \theta_0)^\a A_{\a\b} (\p + \iu \theta_0)^\b  u,v\big)_{L_2(\square)}
\\
&=\sum\limits_{j=1}^{m} (\cB_{2m-j}u,\cB_{j-1} v)_{L_2(\g_l)}
+ \sum\limits_{
\genfrac{}{}{0 pt}{}{\a,\b\in \mathds{Z}_+^d}{|\a|,|\b|\leqslant m}
} \big( A_{\a\b} (\p + \iu \theta_0)^\b  u, (\p + \iu \theta_0)^\a v\big)_{L_2(\square)},
\end{aligned}
\end{equation}
where $\cB_j=\cB_j(x,\p)$ are linear differential operators of order at most $j$ with continuous coefficients defined on $\overline{\g_l}$.
Since the functions $A_{\a\b}$ are $\square$-periodic, the above formula remains true if we replace $\square$ by $\square_k$.
We make the following assumption:

\begin{enumerate}\def\theenumi{(A\arabic{enumi})}
\item \label{A1} There are real-valued functions  $b_j\in C(\overline{\p\square})$ such that
\begin{equation}\label{2.25}
b_j=\frac{\cB_{2m-j}\psi_0}{\cB_{j-1}\psi_0} \chi_{\{\cB_{j-1}\psi_0\neq0\}}
\quad \text{ on} \quad  \g_l,\quad j=1,\ldots,m,
\end{equation}
where $\chi_\natural$ is the characteristic function of a set $\natural$.
Assume also that the coefficients of the boundary operators
$\cB_{2m-j}-b_j\cB_{j-1}$ belong to $C^j(\g_l)$.
\end{enumerate}

On $\Ho^m(\square,\g_l)$  we introduce the sesquilinear form
\begin{equation}\label{2.26}
\widehat{\fm}_0(u,v):= \sum\limits_{
\genfrac{}{}{0 pt}{}{\a,\b\in \mathds{Z}_+^d}{|\a|,|\b|\leqslant m}
} \big( A_{\a\b} (\p + \iu \theta_0)^\b  u, (\p + \iu \theta_0)^\a v\big)_{L_2(\square)} + \sum\limits_{j=1}^{m} (b_j \cB_{j-1}u,\cB_{j-1}v)_{L_2(\g_l)}.
\end{equation}
Thanks to conditions (\ref{2.1}) and assumption \ref{A1}, this form is symmetric, lower-semibounded and closed.
By $\widehat{\Op}_0$ we denote the self-adjoint operator in $L_2(\square)$ associated with the form $\widehat{\fm}_0$.
This is the operator
\begin{equation}\label{2.27}
\widehat{\Op}_0= \sum\limits_{
\genfrac{}{}{0 pt}{}{\a,\b\in \mathds{Z}_+^d}{|\a|,|\b|\leqslant m}
} (-1)^{|\a|}(\p + \iu \theta_0)^\a A_{\a\b} (\p + \iu \theta_0)^\b,
\end{equation}
in $L_2(\square)$ subject to the Dirichlet condition on $\g_\Pi$ and to the condition
\begin{equation}\label{2.28}
(\cB_{2m-j}-b_j \cB_{j-1})u=0\quad \text{on}\quad \g_l.
\end{equation}
The domain of $\widehat{\Op}_0$ consists of the functions in $\Ho^{2m}(\square,\g_{\Pi})$
satisfying the above boundary conditions on $\g_l$.

It is straightforward to check that $\psi_0$ is an eigenfunction of $\widehat{\Op}_0$ associated with the eigenvalue $\L_0$. Moreover, we make one more assumption
\begin{enumerate}\def\theenumi{(A\arabic{enumi})}\setcounter{enumi}{1}
\item \label{A2} The bottom of the spectrum of $\widehat{\Op}_0$ is a simple eigenvalue  and equal to $\L_0$.
\end{enumerate}

Let $\widehat{\psi}_1\in\Dom(\widehat{\Op}_0)$ be the solution to the equation
\begin{equation*}
\big(\widehat{\Op}_0-\L_0\big)\widehat{\psi}_1=-e^{-\iu \theta_0 x}\cL_1e^{\iu \theta_0 x}\psi_0 + \L_1\psi_0
\end{equation*}
such that this solution is orthogonal to $\psi_0$ in $L_2(\square)$. We denote
\begin{equation}\label{2.29}
\widehat{\L}_2:=(\cL_2e^{ \iu \theta_0 x}\psi_0,e^{ \iu \theta_0 x}\psi_0)_{L_2(\square)} + (\cL_1 e^{ \iu \theta_0 x}\widehat{\psi}_1,e^{ \iu \theta_0 x}\psi_0)_{L_2(\square)}.
\end{equation}

The main result concerning a lower bound is provided by the following theorem.

\begin{theorem}\label{th2.3}
Assume \ref{A1} and \ref{A2}. For all sufficiently small $\e\geqslant 0$ the estimate
\begin{equation}\label{2.30}
\inf\Sigma_\e\geqslant 
\L_0+\e s_* \L_1+\e^2 s_*^2\widehat{\L}_2
 - C\e^3
\end{equation}
holds true, where $C$ is a non-negative constant independent of $\e$.
\end{theorem}

An immediate corollary of Theorems~\ref{th2.2} and \ref{th2.3} is that under Assumptions~\ref{A1},~\ref{A2} we have the asymptotics
\begin{equation}\label{2.31}
\inf\spec(\Op_\e(\om))=\L_0+\e s_*\L_1+O(\e^2)\quad\PP-a.s.
\end{equation}
{Moreover, if $\widehat{\L}_2=\L_2$ for at least one of the minimizing triples $(\theta_0,i_0,s_*)$, the asymptotics is even more precise}
\begin{equation}\label{2.32}
\inf\spec(\Op_\e(\om))=\L_0+\e s_*\L_1+\e^2 s_*^2\L_2+O(\e^3)\quad\PP-a.s.
\end{equation}
Note that the additional assumptions~\ref{A1},~\ref{A2} required for the lower bound (\ref{2.30}) are not very restrictive.
A simple example is provided by the  operator $\Op_0$ in $\Pi=\mathds{R}^d$ in the case $A_{\a\b}=0$ for $|\a|<m$ or $|\b|<m$. Then the ellipticity condition in (\ref{2.1}) implies that $(\Op_0 u,u)_{L_2(\mathds{R}^d)}\geqslant 0$.
The lowest eigenvalue of the operator $\Op_0(\theta)$ attains its minimum $\L_0=0$ for $\theta=0$; the associated eigenfunction $\psi_0$ is just a constant: $\psi_0=|\square|^{-\frac{1}{2}} \mathbf{1}$, where $\mathbf{1}$ stands for the function being $1$ everywhere in $\square$.  This eigenvalue is simple.
All the associated functions $b_j$ vanish identically  and the form $\widehat{\fm}_0$ is given by the formula
\begin{equation*}
\widehat{\fm}_0(u,v)= \sum\limits_{ \genfrac{}{}{0 pt}{}{\a,\b\in\mathds{Z}_+^d}{|\a|,|\b|=m}} (A_{\a\b} \p^\b u, \p^\a v )_{L_2(\square)}.
\end{equation*}
The quadratic form $\widehat{\fm}_0(u,u)$ is non-negative and its value vanishes only for the  function $\psi_0$.
It means that in this case both Assumptions~\ref{A1},~\ref{A2} are satisfied.
The constant  $\L_1$ is given by the formula
\begin{equation*}
\L_1=\frac{1}{|\square|} (\cL_1\mathbf{1},\mathbf{1})_{L_2(\square)}.
\end{equation*}
The number $s_*$ is fixed by the inequalities
\begin{equation*}
s_*\L_1\leqslant s_+\L_1,\quad s_*\L_1\leqslant s_-\L_1.
\end{equation*}
The asymptotics (\ref{2.31}) holds true.

It should be also mentioned that we choose Dirichlet boundary condition (\ref{2.5}) just for simplicity.
The type of  boundary condition is not used in our proofs and the Dirichlet condition can be replaced by
any other boundary condition ensuring the self-adjointness of the operator $\Op_0$ on an appropriate subspace of $W_2^{2m}(\Pi)$.
The operators $\cL_i$ are to be symmetric on functions in $W_2^{2m}(\square)$ satisfying the boundary condition on $\g_\Pi$ involved in the definition of $\Op_0$. Then the main results remain true.

\section{Examples}

In this section we discuss examples of unperturbed operators $\Op_0$ and perturbations $\cL_\e$.
First of all we stress that
all the examples we discuss below are considered in an arbitrary periodic domain $\Pi$.
Classical examples of the differential operator $\Op_0$ are Schr\"odinger operators
\begin{equation*}
\Op_0=-\D+V(x),
\end{equation*}
magnetic Schr\"odinger operators
\begin{equation*}
\Op_0=(\iu\nabla+A)^2,
\end{equation*}
operators with a variable metric
\begin{equation*}
\Op_0=-\sum\limits_{i,j=1}^{d} \frac{\p}{\p x_i } A_{ij} \frac{\p}{\p x_j},
\end{equation*}
or the bi-Laplacian
\begin{equation*}
\Op_0=\Delta^2.
\end{equation*}
In the above examples all the coefficients $V$, $A$, $A_{ij}$ are supposed to be periodic  with respect to   $\G$ and smooth enough.

One more class of examples are Schr\"odinger operators with $\d$-interaction on a periodic system of disjoint bounded manifolds. Namely, let $M_0\subset \square$ be a given $C^3$-manifold
{of codimension one without boundary}.
By $M$ we denote the periodic system of manifolds obtained by the shifts of $M_0$ along $\G$. We consider the operator
\begin{equation}\label{5.1}
\Op_0=-\D+V\quad \text{in}\quad \Pi
\end{equation}
subject to the Dirichlet boundary condition on $\p \Pi$ and to the boundary conditions
\begin{equation}\label{5.2}
[u]_M=0,\quad \left[\frac{\p u}{\p\nu}\right]_M=b u|_{M}=0,
\end{equation}
where $[\cdot]_M$ denotes the jump of a function on $M$, $\nu$ is the normal vector to $M$ and $b\in C^2(M)$ is a given function periodic  with respect to   lattice $\G$.
Of course, the above introduced operator does not fit our assumptions
since its domain is not $\Ho^2(\Pi,{\p\Pi})$. But it was shown in \cite[Sect. 8, Ex. 5]{Bo07-AHP}, \cite[Sect. 8, Ex. 5]{Bo07-MPAG}
that there exists an explicit unitary transformation that reduces operator (\ref{5.1}), (\ref{5.2}) into a self-adjoint differential operator with domain $\Ho^2(\Pi,{\p\Pi})$.
This new operator has the same spectrum as the operator introduced in (\ref{5.1}) and  (\ref{5.2}).

The class of examples of operators $\cL_i$  is wider than for the operator $\Op_0$.
The reason is that the operators $\cL_i$ are not necessarily differential ones.
Of course, symmetric differential operators are standard examples for $\cL_i$:
\begin{equation*}
\cL_i=\sum\limits_{
\genfrac{}{}{0 pt}{}{\a,\b\in \mathds{Z}_+^d}{|\a|,|\b|\leqslant m}} (-1)^{|\a|}\p^\a B_{\a\b}\p^b,
\end{equation*}
where $B_{\a\b}$  are given functions  defined on $\square$ satisfying the conditions:
\begin{equation*}
 B_{\a\b}\in W_\infty^{|\a|}(\square),
\quad \overline{B_{\b\a}}=B_{\a\b}.
\end{equation*}
The values on $\g_l$ are to be chosen so that by the integration by parts we get
\begin{equation*}
 \sum\limits_{
\genfrac{}{}{0 pt}{}{\a,\b\in \mathds{Z}_+^d}{|\a|,|\b|\leqslant m}} (-1)^{|\a|} (\p^\a B_{\a\b}\p^b u,v)_{L_2(\square)}=\sum\limits_{
\genfrac{}{}{0 pt}{}{\a,\b\in \mathds{Z}_+^d}{|\a|,|\b|\leqslant m}} ( B_{\a\b}\p^b u, \p^\a v)_{L_2(\square)},\quad u,v\in\Ho^{2{m}}(\square,\g_\Pi),
\end{equation*}
and all the boundary integrals over $\g_l$ vanish. In particular, this
includes potential and magnetic fields. In this case our results are applicable to the perturbation by a random potential and a random magnetic field.

Our next example is an integral operator:
\begin{equation*}
(\cL_i u)(x)=\int\limits_{\square} K(x,y)u(y)\di y,
\end{equation*}
where the kernel $K(x,y)$ is such that $\cL_i$ is a symmetric operator. A sufficient condition is, for instance,
\begin{equation*}
K(x,y)=\overline{K(y,x)}\quad\text{and}\quad \int\limits_{\square\times\square} |K(x,y)|^2 \di x\di y<\infty.
\end{equation*}
Also, more complicated examples like integro-differential operators or pseudodifferential operators are possible.

The next class of examples concerns the situation where a considered operator fits our assumptions
after certain transformations.
Here the first example is the above mentioned $\d$-interaction, see (\ref{5.1}), (\ref{5.2}). As $\cL_i$, we can also choose a $\d$-interaction of such type.
One just needs to replace the second boundary condition in (\ref{5.2}) by
\begin{equation*}
\left[\frac{\p u}{\p\nu}\right]_{M_k} + \e \omega_k b u\big|_{M_k}=0,
\end{equation*}
where $M_k$ is the shift of $M_0$ by $k\in\G$. Then we deal with a random $\d$-potential.
And again by applying the approach of \cite[Sect. 8, Ex. 5]{Bo07-AHP}, \cite[Sect. 8, Ex. 5]{Bo07-MPAG}
we can reduce the operator with  a $\d$-interaction to a differential operator with the same spectrum.
This differential operator satisfies the representation (\ref{2.8a}).

The next example is a small random deformation of a boundary. There are various ways how to introduce it, we dwell only on one of them.
Let $\nu$ be the outward normal to $\p\Pi$, $\z$ be the distance measured along $\nu$ and $\rho\in C^2(\g_\Pi)$ be a given function vanishing in the vicinity of $\overline{\g_l}\cap\p\Pi$.
We introduce a domain $\Pi_\e(\om)$ whose boundary is defined as
\begin{equation*}
\p\Pi_\e(\om):=\{x:\, \theta=\e G(\om,x)\},\quad G(\om,x):=\sum\limits_{k\in\G} \omega_k \rho(x-k).
\end{equation*}
In the domain $\Pi_\e(\om)$ we consider the operator
\begin{equation*}
\Op_\e(\om):=\sum\limits_{
\genfrac{}{}{0 pt}{}{\a,\b\in \mathds{Z}_+^d}{|\a|,|\b|\leqslant m}
} (-1)^{|\a|}\p^\a A_{\a\b} \p^\b
\end{equation*}
subject to the Dirichlet condition with coefficients $A_{\a\b}$ satisfying (\ref{2.1}). As in \cite[Sect. 8, Ex. 4]{Bo07-MPAG},
one can construct a mapping of the domain $\Pi_\e(\om)$  onto $\Pi$ and a unitary transformation of $\Op_\e(\om)$
such that
the transformed  perturbed operator satisfies representation (\ref{2.8a}) and has the same spectrum.

Our next two examples are borrowed from \cite{BKSh15}.
Let $\Wl=\Wl(x)$ be a continuous compactly supported function defined on $\mathds{R}^d$, and $\Ws=\Ws(x,\z)$, $\z=(\z_1,\ldots,\z_d)$,
stands for a function in $\square\times\mathds{R}^{d}$
being  1-periodic  with respect to   each of the variables $\z_i$, $i=1,\ldots,d$, having a zero mean
\begin{equation*}
\int\limits_{(0,1)^{d}} \Ws(x,\xi)\di\xi=0\quad \text{for each} \quad x\in\square,
\end{equation*}
and compactly supported  with respect to   $x$:
\begin{equation*}
\supp \Ws(\cdot,\xi)\subseteq Q\subset \square \quad\text{for each}\quad \xi\in\mathds{R}^d,
\end{equation*}
where $Q$ is a some fixed set. We assume the following smoothness for the function $\Ws$:
\begin{equation*}
\frac{\p^{|\a|+|\b|}\Ws}{\p x^\a \p\xi^\b} \in C(\square\times\mathds{R}^d),\quad \a,\b\in\ZZ_+^d,\quad |\a|\leqslant 3, \quad |\b|\leqslant 1.
\end{equation*}
 We let
\begin{align*}
&\Wloc(x',\e):=\e^{-a} \Wl\left(\frac{x'}{\e}\right), \quad
 \Wosc(x,\e):=\e^{-a} \Ws\left(x,\frac{x}{\e}\right),
 \\
 & \Wloc(x',0):=0,\qquad \Wosc(x',0):=0,
\end{align*}
where $0\leqslant a<1$ is a given number.
The potential $\Wloc$ has {large} values due to the presence of the factor $\e^{-a}$
but is localized on a set of diameter $\e$. The potential $\Wosc$ has also {large} values,
is localized on set $Q$ and oscillates fast due to the presence of $\frac{x}{\e}$ in its definition.

The perturbed operators are introduced as
\begin{equation*}
\Op_\e(\om):=-\D+ \sum\limits_{k\in\G} \Wloc(\cdot-k,\e\omega_k),\quad \Op_\e(\om):=-\D+\sum\limits_{k\in\G} \Wosc(\cdot-k,\e \omega_k).
\end{equation*}
It was shown in  \cite{BKSh15} that there exist unitary transformations reducing each of the above perturbed operators to operators satisfying representation (\ref{2.8a}).

We also stress that any combination of the above perturbed operators is also covered by out general results. For instance, one can consider a combination of a random magnetic field with a random deformation of a boundary, a random integral operator and a $\d$-interaction, a fast oscillating potential and a potential localized on a small set, etc.

\section{Proof of Theorem~\ref{th2.1}}
To prove statement (\ref{2.9}), we follow the ideas of \cite{KirschM-82a}.
In accordance with Theorem~1 in that paper, it is sufficient to check the following conditions:
\begin{enumerate}
\item \textsl{Measurability.}
  For each $z\in\mathds{C}\setminus\mathds{R}$ the mapping $\Om\ni \{\xi_k\}_{k\in\G}\mapsto (\Op_\e(\xi)-z)^{-1}$ defined on our product probability space $(\Om,\mathfrak{F},\PP)$, where $\mathfrak{F}:=\bigotimes_\G \mathfrak{B}$ is the product $\si$-algebra, is weakly measurable. Namely, for each $u,v\in L_2(\Pi)$ the mapping $\{\xi_k\}_{k\in\G}\mapsto \big((\Op_\e(\xi)-z)^{-1}u,v\big)_{L_2(\Pi)}$ is $\PP$-measurable.
\item \textsl{Homogeneity.}
  There exists an index set $I$, a family of measure preserving transformations $T_i\colon \Om\to\Om$, $i\in I$,
  and a family of unitary operators $U_i$, $i\in I$, on $L_2(\Pi)$ satisfying for all $i \in I$
    \begin{equation} \label{eq:covariance}
    \Op_\e( T_i \xi)=U_i \Op_\e(\xi) U_i^*.
    \end{equation}
\item \textsl{Ergodicity.}
    Any $A \in \mathfrak{F}$ such that $T_i^{-1}A=A$ for each $i\in I$ satisfies  $\PP(A)=0$ or $\PP(A)=1$.
\end{enumerate}
We check the measurability following the argument of Proposition~6 in \cite{KirschM-82a}.
First we observe that by the continuity of $\cL_3(t)$ in $t$, the scalar function $t\mapsto (\cL_3(t)u,v)_{L_2(\square)}$ is measurable for each $u\in\Ho^{2m}(\square,\g_\Pi)$, $v\in L_2(\square)$.
Then we represent the resolvent of $\Op_\e(\xi)$ as
\begin{align*}
\big(\Op_\e(\xi)-z\big)^{-1}=&(\Op_0+\cL_\e(\xi)-z)^{-1}
\\
=&(\Op_0-z)^{-\frac{1}{2}} \big(I+(\Op_0-z)^{-\frac{1}{2}}\cL_\e(\xi)(\Op_0-z)^{-\frac{1}{2}}\big)^{-1} (\Op_0-z)^{-\frac{1}{2}}.
\end{align*}

By the lemma following Proposition~6 in \cite{KirschM-82a} the operator $\big(I+(\Op_0-z)^{-\frac{1}{2}}\cL_\e(\xi)(\Op_0-z)^{-\frac{1}{2}}\big)^{-1}$ is measurable and it implies the desired measurability of $\Op_\e(\xi)$.

To show homogeneity, we let
$I:=\G$, $T_i(\{\xi_k\}_{k\in\G}):=\{\xi_{k+i}\}_{k\in\G}$, $U_i:=\cS(i)$.
It is obvious that the operators $\cS(i)$ are unitary and $\cS^*(i)=\cS^{-1}(i)=\cS(-i)$. Employing definition (\ref{2.8a}) of operator $\Op_\e(\om)$ and the obvious identity $\cS(k)\cS(i)=\cS(k+i)$, by straightforward calculations we see that
Condition (\ref{eq:covariance}) is satisfied.
Since the marginal distributions of the product measure $\PP$ are all equal, the mappings $T_i$ are measure preserving, namely, $\PP(T_i^{-1}(A))=\PP(A)$ for each $i\in I$, $A\in\mathcal{F}$.
Moreover, the product structure of $\PP$ implies that the family $T_i, i\in I$ is mixing, and in particular ergodic.
Thus there exists indeed a set $\Sigma_\e\subseteq\RR$ which is almost surely the spectrum of $\Op_\e(\om)$.
In particular,  $\Sigma_\e$ is closed.

In order to prove (\ref{2.10}), we adapt a similar proof in \cite{KirschM-82b} (see Theorems~3,~4 and Lemma~2 in this paper).
Given $N\in\NN$, we denote
\begin{equation*}
\G_N:=\{x\in\mathds{R}^d:\, x=z_1 e_1 +\ldots + z_{d_1}e_{d_1},\,z_i\in \mathds{Z},\, |z_i|\leqslant N\},
\end{equation*}
and $\Pi_N$ is the interior of $\overline{\bigcup\limits_{k\in\G_N}\square_k}$. For $\xi,\om\in\Omega$, $q\in\G$, we let
\begin{equation}\label{3.0}
D_N(\xi,\om,q):=\left(\sum\limits_{k\in\G_N} \|\cL(\e\xi_k)-\cS(q)\cL(\e\omega_k) \cS(-q)\|_{\Ho^{2m}(\square_k,\g_\Pi^k)\to L_2(\square_k)}\right)^{\frac{1}{2}},
\end{equation}
where $\|\cdot\|_{X\to Y}$ stands for the norm of an operator acting from Hilbert space $X$ into Hilbert space $Y$. We introduce the set
\begin{equation}\label{3.4}
B_\xi:=\left\{\om\in\Omega:\, \forall\; N\in \NN \  \forall\; p \in \NN \
\exists\;  q(N,p,\xi)\in\G:\, D_N(\xi,\om,q)<\frac{1}{p} \right\}.
\end{equation}

\begin{lemma}\label{lm3.1}
For each $\xi\in\Om$ we have $\PP(B_\xi)=1$.
\end{lemma}

\begin{proof}
We follow the proof of Lemma~2 in \cite{KirschM-82b}. We let
\begin{equation} \label{eq:BxiNP}
B_{\xi,N,p}:=\left\{\om\in\Omega:\,
\exists\; q(\om)\in\G:\, D_N(\xi,\om,q)<\frac{1}{p}\right\}.
\end{equation}
The definition of $D_N$ implies $D_{N+1}(\xi,\om,p)\geqslant D_N(\xi,\om,p)$ and hence $B_{\xi,N+1,p}\subseteq B_{\xi,N,p}$.
Furthermore $B_{\xi,p+1}\subseteq B_{\xi,p}$.
Set $B_{\xi,p}:=\bigcap\limits_{N=1}^{\infty} B_{\xi,N,p}$
and $B_\xi:=\bigcap\limits_{p=1}^{\infty} B_{\xi,p}$.
Monotone continuity of probability measures implies then
\begin{align*}
& \PP(B_{\xi,p})=\lim\limits_{N\to\infty} \PP(B_{\xi,N,p}),
\quad \PP(B_\xi)=
\lim\limits_{p\to\infty} \PP(B_{\xi,p}).
\end{align*}
Thus, it is sufficient to prove that $\PP(B_{\xi,N,p})=1$ for each $N,p\in \NN$.
Since the set $B_{\xi,N,p}$ is invariant under the shifts $\{T_i\}_{i\in\G}$,
it is by ergodicity sufficient to prove that $\PP(B_{\xi,N,p})>0$.
Obviously
\begin{equation*}
B_{\xi,N,p}^*:=\left\{\om\in\Om:\, D_N(\xi,\om,0)<\frac{1}{p}\right\}
\subseteq B_{\xi,N,p}
\end{equation*}
(just chose $q(\omega)=0$). 
So, it suffices to show $\PP(B_{\xi,N,p}^*)>0$ to complete the proof.

The definition of the operator $\cL(t)$ yields
\begin{equation*}
\cL(t_1)-\cL(t_2)=(t_1-t_2) \cL_1 + (t_1^2-t_2^2) \cL_2 + (t_1^3-t_2^3) \cL_3(t_1) + t_2^3 \big(\cL_3(t_1)-\cL_3(t_2)\big).
\end{equation*}
This identity, definition (\ref{3.0}) of $D_N$ and the continuity of operator $\cL_3(t)$  in $t$ imply immediately
that $D_N(\xi,\om,0)$ tends to zero as $\xi_i-\omega_i\to0$, $i\in\G_N$.
Hence, there exists $\d=\d(\xi,p,N)$ such that $D_N(\xi,\om,0)<\frac{1}{p}$ once $|\xi_i-\omega_i|<\d$, $i\in\G_N$.
In other words,
\begin{equation*}
\left\{\om\in\Om:\, |\xi_i-\omega_i|<\d,\, i\in\G_N\right\}
\subseteq B_{\xi,N,p}^*
\end{equation*}
Since  $\G_N$ is finite, we see that the set on the right
has a positive measure, whenever all $\xi_i, i \in \G_N$ are elements of $\supp \mu$.
It implies that $B_{\xi,N,p}^*$ is of positive $\PP$-measure.
\end{proof}

\begin{lemma}\label{lm3.2}
The identity
\begin{equation*}
\Sigma_\e=\bigcup\limits_{\xi\in\Om} \spec(\Op_\e(\xi))
=\overline{\bigcup\limits_{\xi\in\Om} \spec(\Op_\e(\xi))}
\end{equation*}
holds true.
\end{lemma}

\begin{proof}
We adapt the proof of Theorem~3 in \cite{KirschM-82b}. First of all, by (\ref{2.9}) we have
\begin{equation*}
\Sigma_\e
\subseteq \bigcup\limits_{\xi\in\Om} \spec(\Op_\e(\xi))
\end{equation*}
Let us prove the opposite inclusion.

We fix $\xi\in\Om$ and take $\l\in\spec(\Op_\e(\xi))$. By the Weyl criterion there exists a characteristic sequence $\{\vp_p\}_{p\in\NN}\subset \Dom(\Op_0)$ such that
\begin{equation}\label{3.1}
\|\vp_p\|_{L_2(\Pi)}=1,\quad \|(\Op_\e(\xi)-\l)\vp_p\|_{L_2(\Pi)} \leqslant \frac 1 p, \quad p\in\NN.
\end{equation}
By \cite[Ch. III, Sect. 6, Lm. 6.3]{Ber} we have the estimate
\begin{equation*}
\|\vp_p\|_{W_2^{2m}(\Pi)}
\leqslant C\big (\|(\Op_0+\iu)\vp_p\|_{L_2(\Pi)}
+|\l|+1\big  ),
\end{equation*}
where $C$ is a constant independent of $p$ and $\e$. Hence, the norm $\|\vp_p\|_{W_2^{2m}(\Pi)}$ is bounded uniformly in $p$, and the suprema
\begin{equation*}
\vp_p^*:=\sup\limits_{k\in\G} \|\vp_p\|_{W_2^{2m}(\square_k)},
\quad
\vp^*:=\sup\limits_{p\in\NN} \vp_p^*
\end{equation*}
{are} well-defined  and finite.
Pick $\om$ from the full measure set
$$
\Om^*:=B_\xi\cap \{\om\in\Om:\, \spec(\Op_\e(\om))=\Sigma_\e\}
\subseteq B_{\xi,N_p,p}
$$
and
let us show that $\psi_{p}:=\cS(-q(N_p,p,\om))\vp_p$ is a characteristic sequence for $\Op_\e(\om)$ at $\l$.
Here $N_p$ is a number fixed by the inequality
\begin{equation}\label{3.2}
\|\varphi_p\|_{W_2^{2m}(\Pi\setminus\Pi_{N_p})}\leqslant \frac{1}{p},
\end{equation}
while  $q$ comes from definition (\ref{eq:BxiNP}) of the set $B_{\xi,N_p,p}$.
Employing definition (\ref{3.4}) of the set $B_p$, (\ref{3.1}), (\ref{3.2}), by straightforward calculation we obtain
\begin{align*}
\|&(\Op_\e(\om)-\l)\psi_{p}\|_{L_2(\Pi)} = \|(\Op_0 + \cS(-q)\cL_\e(\om)\cS(q)-\l)\vp_p\|_{L_2(\Pi)}
\\
&\leqslant \|(\Op_\e(\xi)-\l)\vp_p\|_{L_2(\Pi)}  + \|(\cS(-q)\cL_\e(\om)\cS(q)-\cL_\e(\xi))\vp_p\|_{L_2(\Pi)}
\\
&\leqslant  \frac{1}{p} + \|(\cS(-q)\cL_\e(\om)\cS(q)-\cL_\e(\xi))\vp_p\|_{L_2(\Pi_{N_p})}
\\
&\hphantom{\leqslant}+ \|(\cS(-q)\cL_\e(\om)\cS(q)-\cL_\e(\xi))\vp_p\|_{L_2(\Pi\setminus\Pi_{N_p})}
\\
&\leqslant \frac{1}{p} + D_{N_p} (\xi,\om,q) \vp^* + C\|\vp_p\|_{W_2^{2m}(\Pi\setminus\Pi_{N_p})}
\end{align*}
where the constant $C$ is independent of $p$.
Since the chosen $\omega$ is in particular in $B_{\xi,N_p,p}$ and by the choice (\ref{3.2})
\begin{align*}
\frac{1}{p} + D_{N_p} (\xi,\om,q) \vp^* + C\|\vp_p\|_{W_2^{2m}(\Pi\setminus\Pi_{N_p})}
\leqslant \frac{1}{p} + \frac{\vp^* }{p} + \frac{C}{p}.
\end{align*}
Hence, $\psi_{p}$ is indeed a characteristic sequence for $\Op_\e(\omega)-\l$, i.e.~$\l \in\spec(\Op_\e(\omega))$ 
for any $\omega \in \Omega^*$.
Therefore, $\spec(\Op_\e(\xi))\subseteq\Sigma_\e$ for each $\xi\in\Om$. This completes the proof of the lemma.
\end{proof}

In view of the above lemma, we can obtain $\Sigma_\e$ as the union of the spectra $\Op_\e(\xi)$ over all $\xi\in\Om$.
We follow the proof of Theorem~4 in \cite{KirschM-82b} to prove that this union can be taken over $2^k\G$-periodic configurations $\xi$ only.

Given $N\in\NN$, $\xi\in\Om$, we introduce the $2^N\G$-periodic sequence $\xi^{(N)}\in\Om$ as follows: $\xi_k^{(N)}=\xi_k$, $k\in\G_{2^N}$, and for $k\in\G\setminus\G_{2^N}$ we define other terms of $\xi^{(N)}$ by the periodic continuation. If we prove that $\Op_\e(\xi^{(N)})$ converges to $\Op_\e(\xi)$ as $N\to\infty$ in the strong resolvent sense, it will imply $\spec(\Op_\e(\xi))\subseteq\overline{\bigcup\limits_{N=1}^{\infty} \spec(\Op_\e(\xi^{(N)}))}$, and hence,
\begin{equation*}
\overline{\bigcup\limits_{\xi\in\Om} \spec(\Op_\e(\xi))}=\overline{\bigcup\limits_{N\in \NN}\bigcup\limits_{\xi \ \text{is $2^N \G$-periodic}} \spec\big(\Op_\e(\xi)\big)},
\end{equation*}
which will prove (\ref{2.10}).

The desired strong resolvent convergence means that for each $f\in L_2(\Pi)$ the solution to the equation $\big(\Op_\e(\xi^{(N)})-\iu\big)u^{(N)}=f$ converges in $L_2(\Pi)$ to the solution of the equation $\big(\Op_\e(\xi)-\iu\big)u=f$. These equations yield
\begin{equation*}
\big(\Op_\e(\xi^{(N)})-\iu\big)(u-u^{(N)})=\big(\cL_\e(\xi^{(N)})-\cL_\e(\xi)\big)u.
\end{equation*}
Hence, by the definition of $\xi^{(N)}$ and the estimate for the resolvent of a self-adjoint operator
\begin{equation*}
\|u-u^{(N)}\|_{L_2(\Pi)} \leqslant \big\|\big(\cL_\e(\xi^{(N)})-\cL_\e(\xi)\big)u\big\|_{L_2(\Pi)} \leqslant C\|u\|_{W_2^{2m}(\Pi\setminus\Pi_{2^N})},
\end{equation*}
where $C$ is a constant independent of $u$ and $N$. Since $u\in W_2^{2m}(\Pi)$, we get $\|u\|_{W_2^{2m}(\Pi\setminus\Pi_{2^N})}\to0$ as $N\to\infty$
and this completes the proof of Theorem~\ref{th2.1}.


\section{Upper bound on $\inf\Sigma_\e$: Proof of Theorem~\ref{th2.2}.}
In order to prove (\ref{2.19}), first  we estimate the resolvent of the unperturbed operator $\Op_0$.
Since this operator is self-adjoint, we have for any $\l$ in the resolvent set
\begin{equation}\label{4.1}
\|(\Op_0-\l)^{-1}\|_{L_2(\Pi)\to L_2(\Pi)}=\frac{1}{\dist(\l,\spec(\Op_0))}.
\end{equation}
We rewrite the resolvent equation $(\Op_0-\l)u=f$ as
\begin{equation}\label{4.2}
(\Op_0-\iu)u=(\l-\iu)u+f.
\end{equation}
Since the operator $(\Op_0-\iu)$ is invertible, by \cite[Ch. I\!I\!I, Sect. 6, Lm. 6.3]{Ber} and (\ref{4.1}), the solution of equation (\ref{4.2}) satisfies the estimate
\begin{equation}\label{4.3}
\|u\|_{W_2^{2m}(\Pi)} \leqslant C\|(\l-\iu)u+f\|_{L_2(\Pi)} \leqslant \frac{C(|\l|+1)}{\dist(\l,\spec(H_0))}\|f\|_{L_2(\Pi)},
\end{equation}
where $C$ is a positive constant independent of $\l$ and $f$.

Consider the resolvent equation for $\Op_\e(\xi)$, $\xi\in\Om$:
\begin{equation*}
(\Op_0+\cL_\e(\xi)-\l)u=f.
\end{equation*}
We can rewrite it as
\begin{equation*}
\big(I+\cL_\e(\xi)(\Op_0-\l)^{-1}\big)(\Op_0-\l)u=f,
\end{equation*}
where $I$ is the identity mapping. Hence, once
\begin{equation}\label{4.4}
\|\cL_\e(\xi)(\Op_0-\l)^{-1}\|<1,
\end{equation}
the resolvent of $\Op_\e$ is well-defined and  is given by the formula
\begin{equation*}
(\Op_\e(\xi)-\l)^{-1}=(\Op_0-\l)^{-1} \big(I+\cL_\e(\xi)(\Op_0\l)^{-1}\big)^{-1}.
\end{equation*}
In view of the definition of operators $\cL_i$ and estimate (\ref{4.3}), we see that inequality (\ref{4.4}) is satisfied provided

\begin{equation}\label{4.5} \frac{C\e(|\l|+1)}{\dist(\l,\spec(\Op_0))}\sup_{|t|\leqslant t_0} \|\cL(t)\|_{W_2^{2m}(\square) \to L_2(\square)}<1,
\end{equation}
where constant $C$ is the same as in (\ref{4.3}). Hence, such $\l$ are in the resolvent set of $\Op_\e(\om)$.
A contraposition of this statement yields (\ref{2.19}).

Given $s\in\supp \mu$, we denote by $\xi^s$ the constant sequence $\{s\}_{k\in\G}$.
It follows from (\ref{2.10}) that $\Sigma_\e\supseteq\spec(\Op_\e(\xi^s))$, and hence, by the minimax principle
\begin{equation}\label{4.6}
\inf\Sigma_\e\leqslant \inf\spec(\Op_\e(\xi^s))=\inf\limits_{\genfrac{}{}{0 pt}{}{u\in\Dom(\Op_\e(\xi^s))}{u\not=0}} \frac{\big(\Op_\e(\xi^s)u,u\big)_{L_2(\Pi)}}{\|u\|_{L_2(\Pi)}^2}.
\end{equation}
Let $\phi^\e(x,s):=\psi_0(x)+\e s\psi_1(x)$. Since both functions $\psi_0$ and $\psi_1$ satisfy periodic boundary conditions on $\g_l$, we extend $\phi_s^\e$ periodically to $\Pi$ keeping the same notation for the extension.
By $\chi_p\colon \Pi \to [0,1]$ we denote an infinitely differentiable function being one in $\Pi_p$, vanishing in $\Pi\setminus \Pi_{p+2}$ and satisfying the estimates
\begin{equation}\label{4.7}
\left|\frac{\p^\a \chi_p}{\p x^\a}\right| \leqslant C_\a 
\quad \text{in}\quad \overline{\Pi_{p+2}\setminus\Pi_p},
\end{equation}
where $C_\a$ is a positive constant independent of $p$ and $x$.
We also suppose that $\chi_p$ in fact depends only on the $d_1$ coordinates in $S$.
More precisely, for each pair $x,\, \widetilde{x}\in\Pi$ such that $x-\widetilde{x}$ is orthogonal to $S$, we have $\chi_p(x)=\chi_p(\widetilde{x})$.

In view of the above described properties of $\phi^\e$ and $\chi_p$, the function
\begin{equation*}
u_p^\e(x):=e^{\iu\theta_0x}\phi^\e(x,s)\chi_p(x)
\end{equation*}
belongs to the domain of $\Op_0$ and therefore, $u_p^\e\in \Dom(\Op_\e(\xi^s))$. We choose $u_p^\e$ as the test function in (\ref{4.6}) to obtain
\begin{equation}\label{4.8}
\inf\Sigma_\e\leqslant \frac{\big(\Op_\e(\xi^s)u_p^\e,u_p^\e\big)_{L_2(\Pi)}}{\|u_p^\e\|_{L_2(\Pi)}^2}.
\end{equation}

Let us calculate the right hand side of this inequality. It is clear that
\begin{equation}\label{4.9}
\frac{\big(\Op_\e(\xi^s)u_p^\e,u_p^\e\big)_{L_2(\Pi)}}{\|u_p^\e\|_{L_2(\Pi)}^2}=
\frac{\big((\Op_0(\theta_0)+e^{-\iu \theta_0 x} \cL_\e(\xi^s) e^{\iu \theta_0 x})\phi^\e \chi_p,\phi^\e \chi_p\big)_{L_2(\Pi)}}{\|\phi^\e\chi_p\|_{L_2(\Pi)}^2}.
\end{equation}

It follows from the definition of $\phi^\e$ that for each $k\in\G$ the identities
\begin{equation}\label{4.10}
\begin{aligned}
&\|\phi^\e\|_{L_2(\square_k)}^2=1+\e^2s^2\|\psi_1\|_{L_2(\square)}^2,
\\
&\|\phi^\e\|_{W_2^{2m}(\square_k)}^2=\|\psi_0\|_{W_2^{2m}(\square)}^2
+2\e s\RE(\psi_0,\psi_1)_{W_2^{2m}(\square)} + \e^2 s^2 \|\psi_1\|_{W_2^{2m}(\square)}^2
\end{aligned}
\end{equation}
{hold true.}
These identities and the above described properties of $\chi_p$ imply
\begin{align}
&\|\phi^\e\chi_p\|_{L_2(\Pi)}^2=\|\phi^\e\|_{L_2(\Pi_p)}^2 + \|\phi^\e\chi_p\|_{L_2(\Pi_{p+2}\setminus\Pi_p)}^2,\nonumber
\\
&\|\phi^\e\|_{L_2(\Pi_p)}^2\geqslant C p^{d_1},
\quad
\|\phi^\e\|_{L_2(\Pi_{p+2}\setminus\Pi_p)}^2\leqslant C p^{d_1-1}, \label{4.11}
\end{align}
where symbol $C$ stands for inessential constants independent of $p$, $\e$ and $s$.

The boundedness of $\cL_1, \cL_1, \cL_2, \cL_3(t) : W_2^{2m}(\square)\to L_2(\square)$ and (\ref{4.7}) yield
\begin{equation}\label{4.12}
\big\|\big(\Op_0(\theta_0) + e^{-\iu \theta_0 x} \cL_\e(\xi^s) e^{\iu \theta_0 x}\big)\phi^\e \chi_p \big\|_{L_2(\Pi_{p+2}\setminus\Pi_p)} \leqslant C \|\phi^\e\|_{W_2^{2m}(\Pi_{p+2}\setminus\Pi_p)}\leqslant C p^{d_1-1},
\end{equation}
where $C$ is a constant independent of $p$, $\e$ and $s$.

It is clear that
\begin{align}\label{4.13}
& \frac{\big((\Op_0(\theta_0)+e^{-\iu \theta_0 x} \cL_\e(\xi^s) e^{\iu \theta_0 x})\phi^\e \chi_p,\phi^\e \chi_p\big)_{L_2(\Pi)}}{\|\phi^\e\chi_p\|_{L_2(\Pi)}^2}= T_p^1(\e,s) + T_p^2(\e, s),
\\
&T_p^1(\e,s):= \frac{\big((\Op_0(\theta_0)+e^{-\iu \theta_0 x} \cL_\e(\xi^s) e^{\iu \theta_0 x})\phi^\e,\phi^\e \big)_{L_2(\Pi_p)}}{\|\phi^\e\chi_p\|_{L_2(\Pi_{p+2})}^2},\nonumber
\\
&T_p^2(\e,s):= \frac{\big((\Op_0(\theta_0)+e^{-\iu \theta_0 x} \cL_\e(\xi^s) e^{\iu \tau_0 x})\phi^\e\chi_p,\phi^\e \chi_p \big)_{L_2(\Pi_{p+2}\setminus\Pi_p)}}{\|\phi^\e\chi_p\|_{L_2(\Pi_{p+2})}^2}.
\end{align}
By (\ref{4.11}), (\ref{4.12}) we obtain
\begin{equation}\label{4.14}
|T_p^2(\e,s)|\leqslant C p^{-1},
\end{equation}
where the constant $C$ is independent of $p$, $\e$, and $s$, and
\begin{equation}\label{4.15}
\lim\limits_{p\to+\infty}  T_p^1(\e,s) = \frac{\big((\Op_0(\tau_0)+e^{-\iu \theta_0 x} \cL_\e(\om^s) e^{\iu \theta_0 x})\phi^\e,\phi^\e \big)_{L_2(\square)}}{\|\phi^\e\|_{L_2(\square)}^2}.
\end{equation}
Together with (\ref{4.8}), (\ref{4.9}) it yields
\begin{equation}\label{4.16}
\inf \Sigma_\e\leqslant \frac{\big((\Op_0(\theta_0)+e^{-\iu \theta_0 x} \cL_\e(\om^s) e^{\iu \theta_0 x})\phi^\e,\phi^\e \big)_{L_2(\square)}}{\|\phi^\e\|_{L_2(\square)}^2}.
\end{equation}
Equation (\ref{2.17}) and the eigenvalue equation for $\psi_0$ imply that
\begin{align*}
\big(\Op_0(\theta_0) + e^{-\iu \theta_0 x} \cL_\e(\om^s) e^{\iu \theta_0 x}\big)\phi^\e=&\L_0\phi^\e + \e s \L_1 \psi_0
\\
&+ \e^2 s^2 e^{-\iu \theta_0 x} (\cL_2 + \e s \cL_3(\e s) ) e^{\iu \theta_0 x} \phi_\e
\\
 &+ \e^2 s^2 e^{-\iu \theta_0 x} \cL_1 e^{\iu \theta_0 x} \psi_1.
\end{align*}
Substituting this identity into (\ref{4.16}), we get
\begin{equation}\label{4.17}
\begin{aligned}
\inf\Sigma_\e\leqslant &\L_0+ \frac{1}{\|\phi^\e\|_{L_2(\square)}^2} \Big(
\e s\L_1(\psi_0,\phi^\e)_{L_2(\square)} + \e^2 s^2 (\cL_2 e^{\iu\tht_0 x}\phi^\e, e^{\iu\tht_0 x}\phi^\e)_{L_2(\square)}
\\
&+\e^2 s^2 (\cL_1 e^{\iu\tht_0 x}\psi_1,e^{\iu\tht_0 x}\phi^\e)_{L_2(\square)}
+\e^3 s^3 (\cL_3(\e s) e^{\iu\tht_0 x}\phi^\e,e^{\iu\tht_0 x}\phi^\e)_{L_2(\square)}.
\Big)
\end{aligned}
\end{equation}
We employ the identities
\begin{equation*}
(\psi_0,\psi_1)_{L_2(\square)}=0,\quad \|\phi^\e\|_{L_2(\square)}^2=1+\e^2 s^2 \|\psi_1\|_{L_2(\square)}^2
\end{equation*}
and definitions (\ref{2.16}), (\ref{2.18}) of $\L_1$, $\L_2$
to check that
\begin{equation*}
 \e s\L_1 (\psi_0,\phi^\e)_{L_2(\square)}=\e s \L_1\|\phi^\e\|_{L_2(\square)}^2 - \e^3 s^3 \L_1 \|\psi_1\|_{L_2(\square)}^2
\end{equation*}
and
\begin{align*}
(\cL_2 e^{\iu\tht_0 x}\phi^\e,&e^{\iu\tht_0 x}\phi^\e)_{L_2(\square)} + (\cL_1 e^{\iu\tht_0 x}\psi_1,e^{\iu\tht_0 x}\phi^\e)_{L_2(\square)}
\\
=& \L_2\big(\|\phi^\e\|_{L_2(\square)}^2 - \e^2 s^2 \|\psi_1\|_{L_2}^2\big)
+ 2\e s \RE (\cL_2 e^{\iu\tht_0 x}\psi_0,e^{\iu\tht_0 x}\psi_1)_{L_2(\square)}
\\
&  + \e^2 s^2 (\cL_2 e^{\iu\tht_0 x}\psi_1,e^{\iu\tht_0 x}\psi_1)_{L_2(\square)} + \e s (\cL_1 e^{\iu\tht_0 x}\psi_1,e^{\iu\tht_0 x}\psi_1)_{L_2(\square)}.
\end{align*}
Together with (\ref{4.17}) it yields
\begin{equation*}
\inf \Sigma_\e \leqslant \L_0 + \e s \L_1 + \e^2 s^2 \L_2 + \frac{\e^3 s^3 \L_3(\e s)}{1+\e^2 s^2 \|\psi_1\|_{L_2(\square)}^2}.
\end{equation*}
And thanks to (\ref{2.9}) it proves (\ref{2.20}). The proof of Theorem~\ref{th2.2} is complete.

\medskip

In this section we also prove Corollary~\ref{c:1}.

\begin{proof}[Proof of Corollary \ref{c:1}]
Since {the function $e^{\iu \theta_0 x}\psi_0(x)$}
is in the domain of $\cL_2$,
\begin{equation*}
(\cL_2e^{\iu \theta_0 x}\psi_0,e^{\iu \theta_0 x}\psi_0)_{L_2(\square)} \leqslant 0.
\end{equation*}
For  $\L_1=0$ {we have}
\begin{align*}
(\cL_1e^{\iu \theta_0 x}\psi_1,e^{\iu \theta_0 x}\psi_0)_{L_2(\square)}
= &-(\Op_0(\theta_0) \psi_1,\psi_1)_{L_2(\square)}
= -(E_2-\L_0) \|\psi_1\|^2_{L_2(\square)}
\\
\leqslant &-c_1 \|\psi_1\|^2_{L_2(\square)},
\end{align*}
where $E_2=\spec(\Op_0(\theta))\setminus\{E_0(\theta)\}$ and $c_1$ the spectral gap.
Thus,
\[
\L_2
\leqslant - c_1 \|\psi_1\|^2 + (\cL_2e^{\iu \theta_0 x}\psi_0,e^{\iu \theta_0 x}\psi_0)_{L_2(\square)}
\leqslant - c_1 \|\psi_1\|^2.
\]
\end{proof}

\section{Lower bound on $\inf\Sigma_\e$: Proof of Theorem~\ref{th2.3}}
Let $\g_l^\pm$ be a pair of opposite faces in $\g_l$, namely,
\begin{equation*}
\g_l^-:=\p\square\cap\p\square_{-e_i},\quad \g_l^+:=\p\square\cap\p\square_{e_i}
\end{equation*}
for some $1\leqslant i\leqslant d_1$.
We recall that $e_1,\ldots,e_{d_1}$ is the basis of lattice $\G$, and $\square_{\pm e_i}$ is just $\square_k$ with $k=\pm e_i$. Let us show that
\begin{equation}\label{4.18}
b_j\big|_{\g_l^-}=-b_j\big|_{\g_l^+}
\end{equation}
for any choice of $1\leqslant i\leqslant d_1$.
We first observe that by the periodicity of the functions $A_{\a\b}$ and $\psi_0$ we have
\begin{equation}\label{4.18a}
\Big|\cB_j\psi_0\big|_{\g_l^-}\Big|=\Big|\cB_j\psi_0\big|_{\g_l^+}\Big|.
\end{equation}
For each $u\in \Dom(\Op_0)$
with compact support we have
\begin{align*}
 \sum\limits_{
\genfrac{}{}{0 pt}{}{\a,\b\in \mathds{Z}_+^d}{|\a|,|\b|\leqslant m}
} (-1)^{|\a|} & \int\limits_{\Pi} \overline{\psi_0}   (\p + i \theta)^\a A_{\a\b} (\p + i \theta)^\b u \di x
\\
&= \sum\limits_{
\genfrac{}{}{0 pt}{}{\a,\b\in \mathds{Z}_+^d}{|\a|,|\b|\leqslant m}
}   \int\limits_{\Pi} A_{\a\b} (\p + i \theta)^\b u\, \overline{(\p + i \theta)^\a\psi_0}\di x
\end{align*}
and by (\ref{2.23}) we also get
\begin{align*}
\sum\limits_{
\genfrac{}{}{0 pt}{}{\a,\b\in \mathds{Z}_+^d}{|\a|,|\b|\leqslant m}
} &(-1)^{|\a|} \int\limits_{\Pi} \overline{\psi_0}   (\p + \iu \theta_0)^\a A_{\a\b} (\p + \iu \theta_0)^\b u \di x
\\
=& \sum\limits_{k\in\G} \sum\limits_{j=1}^{m} (\cB_{2m-j}\cS(-k)u,\cB_j \psi_0)_{L_2(\g_l)}
 \\
 &+ \sum\limits_{k\in\G}  \sum\limits_{
\genfrac{}{}{0 pt}{}{\a,\b\in \mathds{Z}_+^d}{|\a|,|\b|\leqslant m}
}  \big( A_{\a\b} (\p + \iu \theta_0)^\b u, \overline{(\p + \iu \theta_0)^\a\psi_0}\big)_{L_2(\square_k)}
\\
=& \sum\limits_{k\in\G} \sum\limits_{j=1}^{m} (\cB_{2m-j}\cS(-k)u,\cB_j \psi_0)_{L_2(\g_l)}
 \\
 &+   \sum\limits_{
\genfrac{}{}{0 pt}{}{\a,\b\in \mathds{Z}_+^d}{|\a|,|\b|\leqslant m}
}  \int\limits_{\Pi} A_{\a\b} (\p + i \theta)^\b u\, \overline{(\p + i \theta)^\a\psi_0}\di x.
\end{align*}
Hence,
\begin{equation*}
\sum\limits_{k\in\G} \sum\limits_{j=1}^{m} (\cB_{2m-j}\cS(-k)u,\cB_j \psi_0)_{L_2(\g_l)}=0.
\end{equation*}
Since $u$ is arbitrary and $\psi_0$ is periodic, the above identity is possible only if for all $1\leqslant i\leqslant d_1$
\begin{align*}
\cB_{2m-j} \cS(-k)u\Big|_{\g_l^-} \overline{\cB_j\psi_0}\Big|_{\g_l^-}= & -\cB_{2m-j} \cS(-k-e_i)u\Big|_{\g_l^-} \overline{\cB_j\cS(-e_i)\psi_0}\Big|_{\g_l^-}
\\
=&-\cB_{2m-j} \cS(-k )u\Big|_{\g_l^+} \overline{\cB_j\psi_0}\Big|_{\g_l^+}
\end{align*}
for each $k\in\G$. Dividing this identity by  $\Big|\cB_j\psi_0\Big|_{\g_l^-}\Big|^2=\Big|\cB_j\psi_0\Big|_{\g_l^+}\Big|^2$,
 cf.~(\ref{4.18a}), and letting $u=\psi_0$ in the vicinity of $\square_k$, we arrive at (\ref{4.18}).

It follows from (\ref{4.18}) that for each
{$u\in\Dom(\Op_0)$}
\begin{equation*}
\sum\limits_{k\in\G} \sum\limits_{j=1}^{m} (\cB_{2m-j}\cS(-k)u,\cB_j \cS(-k) u)_{L_2(\g_l)} = 0.
\end{equation*}
Employing this identity and the minimax principle, we obtain:
\begin{equation}
\begin{aligned}
\inf\spec\big(\Op_\e(\xi)\big)& = \inf\limits_{\genfrac{}{}{0 pt}{}{
{u\in\Dom(\Op_0)}
}{u\not=0}} \frac{\fm_0(u,u) + (\cL_\e(\xi)u,u)_{L_2(\Pi)}}{\|u\|_{L_2(\Pi)}^2}
\\
 &=  \inf\limits_{\genfrac{}{}{0 pt}{}{
{u\in\Dom(\Op_0)}
}{u\not=0}} \frac{\fm_0(e^{\iu\theta_0x}u,e^{\iu\theta_0x}u) + (\cL_\e(\xi)e^{\iu\theta_0x}u,e^{\iu\theta_0x}u)_{L_2(\Pi)}}{\|u\|_{L_2(\Pi)}^2}
\\
 &=  \inf\limits_{\genfrac{}{}{0 pt}{}{
{u\in\Dom(\Op_0)}
}{u\not=0}}
\frac{1}{\|u\|_{L_2(\Pi)}^2} \Big(
 \fm_0(e^{\iu\theta_0x}u,e^{\iu\theta_0x}u) + ( \cL_\e(\xi)e^{\iu\theta_0x}u,e^{\iu\theta_0x}u)_{L_2(\Pi)}
\\
&\hphantom{= \inf\limits_{\genfrac{}{}{0 pt}{}{
{u\in\Dom(\Op_0)}
}{u\not=0}}\frac{1}{\|u\|_{L_2(\Pi)}^2} \Big(}+ \sum\limits_{k\in\G} \sum\limits_{j=1}^{m} (b_j \cB_{j-1}\cS(-k)u,\cB_{j-1}\cS(-k)u)_{L_2(\g_l)}  \Big)
\\
&= \inf\limits_{\genfrac{}{}{0 pt}{}{
{u\in\Dom(\Op_0)}
}{u\not=0}}
\frac{\sum\limits_{k\in\G} \Big(
\widehat{\fm}_0(u,u)  + \big( \cL(\e\xi_k)e^{\iu\theta_0x}\cS(-k)u,e^{\iu\theta_0x}\cS(-k)u\big)_{L_2(\square)}
\Big)}{\sum\limits_{k\in\G}\|\cS(-k)u\|_{L_2(\square)}^2}.
\end{aligned}\label{4.20}
\end{equation}
Since
\begin{equation*}
\Dom(\Op_0)
\subseteq \bigoplus\limits_{k\in\G} \Ho^{2m}(\square_k,\g_\Pi^k):=\left\{u\in L_2(\Pi):\, u\big|_{\square_k} \in \Ho^{2m}(\square_k,\g_\Pi^k),\,k\in\G
\right\},
\end{equation*}
by (\ref{4.20}) we get
\begin{equation}\label{4.21}
\begin{aligned}
\inf\spec\big(\Op_\e(\xi)\big)
&\geqslant  \inf\limits_{\bigoplus\limits_{k\in\G} \Ho^{2m}(\square_k,\g_\Pi^k)}
\frac{1}{\sum\limits_{k\in\G}\|\cS(-k)u\|_{L_2(\square)}^2}
 \sum\limits_{k\in\G}  \Big(
\widehat{\fm}_0(\cS(-k)u, \cS(-k)u)
\\
&\hphantom{\geqslant  \inf\limits_{\bigoplus\limits_{k\in\G}\Ho^{2m}(\square,\g_\Pi )} }+ \big( \cL(\e\xi_k)e^{\iu\theta_0x}\cS(-k)u,e^{\iu \theta_0 x}\cS(-k)u\big)_{L_2(\square)}
\Big)
\\
&\geqslant \inf\limits_{[s_-,s_+]} \l_\e(s),
\end{aligned}
\end{equation}
where
\begin{equation*}
\l_\e(s):=\inf\limits_{\Ho^{2m}(\square,\g_\Pi) }    \frac{\widehat{\fm}_0(u,u)  +
\big(\cL(\e s)e^{\iu\theta_0x} u, e^{\iu\theta_0x} u\big)_{L_2(\square)}
}{\|u\|_{L_2(\square)}^2}.
\end{equation*}
By the minimax principle, $\l_\e(s)$ is the bottom of the spectrum of the operator $$\widehat{\Op}_0 + e^{-\iu\theta_0 x}\cL(\e s)e^{\iu\theta_0 x}.$$ Assumption \ref{A2} and the $\widehat{\Op}_0$-boundedness of $\cL(\e s)$ yield that $\l_\e(s)$ is a discrete eigenvalue of $\widehat{\Op}_0 + e^{-\iu\theta_0 x}\cL(\e s)e^{\iu\theta_0 x}$ and $\l_\e(s)\to\L_0$ as $\e\to+0$ uniformly in $s\in[s_-,s_+]$. By regular perturbation theory one can easily construct the asymptotic expansion for $\l_\e(s)$:
\begin{equation}\label{4.22}
\l_\e(s)=\L_0+\e s\widehat{\L}_1 + \e^2 s^2 \widehat{\L}_2 + O(\e^3),
\end{equation}
where the estimate for the error term is uniform in $s\in[s_-,s_+]$, \begin{equation*}
\widehat{\L}_1=(e^{-\iu\theta_0 x}\cL_1 e^{\iu\theta_0 x}\psi_0,\psi_0)_{L_2(\square)}=\L_1,
\end{equation*}
and $\widehat{\L}_2$ is given by formula (\ref{2.29}).
The asymptotics (\ref{4.22}), estimate (\ref{4.21}) and the definition of $\l_\e(s)$ imply
(\ref{2.30}).
The proof is complete.

\section*{Acknowledgments}

This work was initiated while the authors were at the Chair of Stochastics, Faculty of Mathematics, of the Technische Universit\"at Chemnitz. It was partially financially supported by the DFG through the project grant \emph{Eindeutige-Fortsetzungsprinzipien und Gleichverteilungseigenschaften von Eigenfunktionen}. The research of D.B. was supported by the grant of Russian Science Foundation no. 14-11-00078.

\end{document}